\documentclass[11pt,a4paper]{amsart}
\usepackage{amsfonts}

\theoremstyle{plain}
\newtheorem{Thm}{Theorem}
\newtheorem{Cor}[Thm]{Corollary}
\newtheorem{Lem}{Lemma}%[section]

\theoremstyle{definition}
\newtheorem{Rem}{Remark}

\newcommand{\N}{\mathbb{N}}
\newcommand{\Z}{\mathbb{Z}}
\newcommand{\R}{\mathbb{R}}
\newcommand{\C}{\mathbb{C}}
\newcommand{\calK}{\mathcal{K}}
\newcommand{\calS}{\mathcal{S}}
\newcommand{\calF}{\mathcal{F}}
\newcommand{\calI}{\mathcal{I}}
\newcommand{\calJ}{\mathcal{J}}
\newcommand{\calP}{\mathcal{P}}
\newcommand{\calQ}{\mathcal{Q}}

\allowdisplaybreaks[3]

\begin{document}
%------------------------------

%\date{June 23, 2011.}

\renewcommand\datename{\textbf{THIS PAPER HAS BEEN PUBLISHED IN}}
\date{\textit{Expo.\@ Math.\@} \textbf{30} (2012), 32--48.} 
% Expositiones Mathematicae
% ISSN: 0723-0869 

\title[Bilinear biorthogonal expansions]{%
Bilinear biorthogonal expansions and the Dunkl kernel on the real line}

\author[L. D. Abreu]{Lu\'{\i}s Daniel Abreu}
\address{CMUC and Departamento de Matem\'atica,
Universidade de Coimbra,
Faculdade de Ci\^{e}ncias e Tecnologia (FCTUC),
3001-454 Coimbra, Portugal}
\email{daniel@mat.uc.pt}
\thanks{Research of the first author supported by CMUC/FCT and FCT post-doctoral grant SFRH/BPD/26078/2005, POCI 2010 and FSE}

\author[\'O. Ciaurri]{\'Oscar Ciaurri}
\address{CIME and Departamento de Matem\'aticas y Computaci\'on,
Universidad de La Rioja, % Edificio J.~L.~Vives,
% Calle Luis de Ulloa s/n, 
26004 Logro\~no, Spain}
\email{oscar.ciaurri@unirioja.es}

\author[J. L. Varona]{Juan Luis Varona}
\address{CIME and Departamento de Matem\'aticas y Computaci\'on,
Universidad de La Rioja, % Edificio J.~L.~Vives,
% Calle Luis de Ulloa s/n, 
26004 Logro\~no, Spain}
\email{jvarona@unirioja.es}
\urladdr{http://www.unirioja.es/cu/jvarona/}
\thanks{Research of the second and third authors supported by grant MTM2009-12740-C03-03 of the DGI}


\begin{abstract}
We study an extension of the classical Paley-Wiener space structure, which is based on bilinear expansions of integral kernels into biorthogonal sequences of functions. The structure includes both sampling expansions and Fourier-Neumann type series as special cases, and it also provides a bilinear expansion for the Dunkl kernel (in the rank $1$ case) which is a Dunkl analogue of Gegenbauer's expansion of the plane wave and the corresponding sampling expansions. In fact, we show how to derive sampling and Fourier-Neumann type expansions from the results related to the bilinear expansion for the Dunkl kernel.
\end{abstract}

%------------------------------
\maketitle
%------------------------------

%% ONLY FOR DRAFTS:
%\tableofcontents

%------------------------------
\section{Introduction}
%------------------------------

The function $K(x,t) = e^{ixt}$ has several well-known bilinear expansion
formulas, for example: the Fourier series expansion,
\begin{equation}
\label{eq:expsamp}
  e^{ixt} = \sum_{n = -\infty}^{\infty}
  \frac{\sin (x-\pi n)}{x-\pi n} \,e^{i\pi nt},
  \qquad t \in [-1,1];
\end{equation}
the expansion in terms of the prolate spheroidal wavefunctions~$\varphi_{n}$,
\begin{equation}
\label{eq:prolate}
  e^{-ixt} = \sqrt{2\pi}\sum_{n = 0}^{\infty}i^{n}
  \lambda_{n}\varphi_{n}(x)\varphi_{n}(t),
\end{equation}
where $\lambda_{n}$ are the square roots of the eigenvalues arising from
the time-band limiting integral equation (see the recent paper~\cite{Zayed});
and Gegenbauer's expansion of the plane wave in Gegenbauer polynomials 
and Bessel functions (see~\cite[\S\,4.8, formula (4.8.3), p.~116]{Ism})
\begin{equation}
\label{eq:geg}
  e^{ixt} = \Gamma(\beta) \left(\frac{x}{2}\right)^{-\beta}
  \sum_{n = 0}^{\infty} i^{n}(\beta+n)J_{\beta+n}(x)C_{n}^{\beta}(t),
  \qquad t \in [-1,1],
\end{equation}
(in the particular case $\beta=0$, this formula is the so-called
Jacobi-Anger identity). Here and in what follows in this paper we 
are using $C_{n}^{\beta}$ to denote the Gegenbauer polynomials 
of order~$\beta$ and $J_\nu$ to denote the Bessel functions of 
order~$\nu$.

Each of the above expansions is associated with important developments in
mathematical analysis. The first one is equivalent to the
Whittaker-Shannon-Kotel'nikov sampling theorem (see \cite[Ch.~2]{Zayed-book}), 
the second one is the prototype of a Mercer kernel \cite{Zayed}, 
and the third one has been the main tool in the diagonalization of 
certain integral operators~\cite{IZ}. 

Our main interest is in expansions of the type~\eqref{eq:geg}. To see why biorthogonality is required, recall the definition of the Paley-Wiener space~$PW$,
\begin{equation*}
  PW = \left\{f\in L^2(\R): f(z) = (2\pi)^{-1/2}\int_{-1}^1 u(t)e^{izt}\, dt ,\ u\in L^2(-1,1)\right\}.
\end{equation*}
We immediately obtain orthogonal expansions for $f\in PW$ by simply integrating equations \eqref{eq:expsamp} and \eqref{eq:prolate}, by using the orthogonality of the exponentials and the prolate spheroidal functions. However, if we try to do the same thing in \eqref{eq:geg}, we must restrict ourselves to the case $\beta=1/2$, when the weight function of the Gegenbauer's polynomials (actually Legendre) is~$1$. Therefore, even in this simple case it is not clear how to expand Paley-Wiener functions into Gegenbauer polynomials or Bessel functions with general parameter~$\beta$, and the biorthogonal formulation serve to introduce the parameter $\beta$ of~\eqref{eq:geg} in a natural way.

To organize the presentation of our ideas, we first construct a structure involving biorthogonal expansions, from which the results are obtained, after explicit evaluation of some integrals. In particular (but not exclusively), we use this structure to analyze the solution of the above mentioned expansion problem and its extension to the Dunkl kernel
\begin{equation*}
  E_{\alpha}(ix) = 2^{\alpha}\Gamma(\alpha+1)
  \left( \frac{J_{\alpha}(x)}{x^{\alpha}}+\frac{J_{\alpha+1}(x)}
  {x^{\alpha+1}}\,xi\right)
\end{equation*}
(as we will see in subsection~\ref{sub:Dunkl}, the Dunkl kernel is used to define the Dunkl transform on the real line similarly to how the kernel $e^{ixt}$ is used to define the Fourier transform),
and so we expand functions in $PW$ and its generalization studied in \cite{CV} in terms of Fourier-Neumann series.
From the following extension of \eqref{eq:geg} to the Dunkl kernel
\begin{equation*}
%\label{eq:Dunklexpans}
  E_\alpha(ixt) = \Gamma(\alpha+\beta+1) \left(\frac{x}{2}\right)^{-\alpha-\beta-1}
  \sum_{n=0}^\infty i^n (\alpha+\beta+n+1)J_{\alpha+\beta+n+1}(x)
  C_n^{(\beta+1/2,\alpha+1/2)}(t),
\end{equation*}
where $C_n^{(\beta+1/2,\alpha+1/2)}$ are the so-called generalized Gegenbauer polynomials, 
we obtain uniformly convergent Fourier-Neumann type expansions
\begin{equation*}
  f(x) = \sum_{n=0}^\infty a_n(f)(\alpha+\beta+n+1)x^{-\alpha-\beta-1}J_{\alpha+\beta+n+1}(x),
\end{equation*}
valid for $f\in PW_\alpha$ (the natural generalization of the Paley-Wiener space as in \cite{CV}), and where
\begin{equation*}
  a_n(f) = 2^{\alpha+\beta+1} \Gamma(\alpha+\beta+1)
  \int_{\R} f(t) \frac{J_{\alpha+\beta+n+1}(t)}{t^{\alpha+\beta+1}} \,d\mu_{\alpha+\beta}(t).
\end{equation*}
Moreover, in some cases, the coefficients $a_n(f)$ are identified as Fourier coefficients.

The paper is organized as follows. In the second section we
describe our problem in abstract terms. First we build the general
setup for bilinear orthogonal expansions and, later, we modify it
to consider biorthogonal sequences in the
expansions (see Theorem~\ref{thm:expbilin}). 
In the third section we describe the results which are
obtained in the case of the Fourier kernel. 
The fourth section studies the expansion associated with the Dunkl kernel
(see Theorem~\ref{thm:Dunkl}), that we think is a new and interesting result; 
moreover, we also show its consequences for the Hankel transform.
In the last section we collect the evaluation of some integrals involving 
special functions which were essential for the paper but could not be found 
in the literature.

%\pagebreak[4]
%------------------------------
\section{Structure}
%------------------------------

%------------------------------
\subsection{Orthogonal expansions}
\label{sub:set-up}
%------------------------------

%%% ALTERNATIVA APARENTEMENTE MAS GENERAL, PERO QUE ES EQUIVALENTE
%We begin with $K(x,t)$, a function of two variables defined on $\Omega
%\times \Omega \subset \R\times \R$, and an interval $I\subset \Omega$.
%Using this function as a kernel, define on $L^{2}(\Omega,d\mu)$, 
%with $d\mu$ a non-negative real measure, an integral
%transformation by
%\begin{equation}
%\label{eq:calK}
%  (\calK f)(t) = \int_{\Omega} f(x)\overline{K(x,t)}\,d\mu(x),
%\end{equation}
%with inverse
%\begin{equation}
%\label{eq:invcalK}
%  (\widetilde{\calK}g)(x) = \int_{\Omega} g(t)K(x,t)\,d\mu(t),
%\end{equation}
%and satisfying the multiplication formula
%\begin{equation}
%\label{eq:multiplication}
%  \int_{\Omega} (\calK f) g\,d\mu = \int_{\Omega} (\calK g)f\,d\mu.
%\end{equation}
%Note that the previous identity implies the following one:
%\begin{equation*}
%  \int_{\Omega} (\widetilde{\calK} f) g\,d\mu
%  = \int_{\Omega} (\widetilde{\calK} g) f\,d\mu.
%\end{equation*}
%Moreover, if in the multiplication formula we take $g = \overline{\calK(f)}$
%and use $\calK(\overline{h}) = \overline{\widetilde{\calK}(h)}$,
%we get $\|\calK f\|_{L^2(\Omega,d\mu)} = \|f\|_{L^2(\Omega,d\mu)}$.
%That is, we guarantee that $\calK$ is an isometry.
%The most common situation is to assume that the kernel satisfies $K(x,t) = K(t,x)$.
%In this way, the multiplication formula follows from Fubini's theorem.
%%% FIN ALTERNATIVA

We begin with $K(x,t)$, a function of two variables defined on $\Omega
\times \Omega \subset \R\times \R$ satisfying
\begin{equation*}
  K(x,t)=K(t,x)
\end{equation*} 
almost everywhere for $(x,t) \in \Omega \times \Omega$, 
and an interval $I\subset \Omega$.
Using this function as a kernel, define on $L^{2}(\Omega,d\mu)$, 
with $d\mu$ a non-negative real measure, an integral
transformation by
\begin{equation}
\label{eq:calK}
  (\calK f)(t) = \int_{\Omega} f(x)\overline{K(x,t)}\,d\mu(x).
\end{equation}
Moreover, we suppose that $\calK$ is invertible and that the inverse is
\begin{equation}
\label{eq:invcalK}
  (\widetilde{\calK}g)(x) = \int_{\Omega} g(t)K(x,t)\,d\mu(t).
\end{equation}
Then, from Fubini's theorem we get the multiplication formula
\begin{equation}
\label{eq:multiplication}
  \int_{\Omega} (\calK f) g\,d\mu = \int_{\Omega} (\calK g)f\,d\mu
\end{equation}
and also
\begin{equation*}
  \int_{\Omega} (\widetilde{\calK} f) g\,d\mu
  = \int_{\Omega} (\widetilde{\calK} g) f\,d\mu.
\end{equation*}
Moreover, if in the multiplication formula we take $g = \overline{\calK(f)}$
and use $\calK(\overline{h}) = \overline{\widetilde{\calK}(h)}$,
we get $\|\calK f\|_{L^2(\Omega,d\mu)} = \|f\|_{L^2(\Omega,d\mu)}$.
% That is, we guarantee that $\calK$ is an isometry.

As usual, it is enough to suppose that the operators
$\calK$ and $\widetilde{\calK}$,
defined by~\eqref{eq:calK} and~\eqref{eq:invcalK}, are
defined on a suitable dense subset of $L^2(\Omega,d\mu)$, and later
extended to the whole $L^2(\Omega,d\mu)$ in the standard way.
Moreover, let us also assume that, as a function of~$t$,
$K(x,t)\chi_I(t)$ belongs to $L^{2}(\Omega,d\mu)$
(or, in other words, $K(x,\cdot) \in L^{2}(I,d\mu)$).
Here, $\chi_{I}$ stands for the characteristic function of~$I$.

Now, let $N$ be a subset of the integers, $\{\phi_{n}\}_{n\in N}$ be
an orthonormal basis of the space $L^{2}(I,d\mu)$ and $\{S_{n}\}_{n\in N}$ be 
a sequence of functions in $L^{2}(\Omega,d\mu)$ such that, for every $n\in N$,
\begin{equation}
\label{eq:calKSn}
  (\calK S_{n})(t) = \chi_{I}(t)\overline{\phi_{n}(t)}
\end{equation}
(notice the small abuse of notation in the use of $\chi_{I}\overline{\phi_{n}}$;
here, $\overline{\phi_{n}}$ is a function that is defined only on $I$, and by
$\chi_{I}\overline{\phi_{n}}$ we mean that we extend this function to $\Omega$ by
having it be the null function on $\Omega \setminus I$; we will use this kind
of notation often in this paper). Consider
the subspace $\calP$ of $L^{2}(\Omega,d\mu)$ constituted by those
functions $f$ such that $\calK f$ vanishes outside of~$I$. This can
also be written as
\begin{equation*}
  \calP = \Bigl\{ f\in L^{2}(\Omega) :
  f(x) = \int_{I} u(t)K(x,t)\,d\mu(t),\ u\in L^{2}(I,d\mu) \Bigr\}.
\end{equation*}
On the one hand, by using that $\calK$ is an isometry,
it follows that $S_{n}$ is a \emph{complete orthonormal} sequence
in~$\calP$. This implies that every function $f$ in $\calP$
has an expansion of the form
\begin{equation}
\label{eq:fcS}
  f(x) = \sum_{n\in N} c_{n}S_{n}(x).
\end{equation}
On the other hand, from \eqref{eq:calKSn} we have
\begin{equation*}
  \widetilde{\calK}(\chi_{I}\overline{\phi_{n}})
  = \widetilde{\calK}(\calK S_{n}) = S_{n}.
\end{equation*}
Consequently, the Fourier coefficients of $K(x,t)$, as a function of $t$, in
the basis $\{\phi_{n}\}_{n\in N}$ on $L^2(I,d\mu)$ are $S_{n}(x)$.
As a result, $K(x,t)$  has the following bilinear expansion formula:
\begin{equation}
\label{eq:Kxt}
  K(x,t) = \sum_{n\in N}S_{n}(x)\phi_{n}(t),
\end{equation}
that must be understood as in $L^2(I,d\mu(t))$ for every $x \in \Omega$.

\begin{Rem}
The reader familiar with sampling theory, in particular with the
generalization due to Kramer (see~\cite{Kramer} or, also, 
\cite[Theorem~3.5]{Zayed-book}), has probably noticed strong resemblances. 
Indeed, Kramer's lemma corresponds to a particular
case of the above situation when there exists a sequence of points
$x_{k}$ such that $S_{n}(x_{k}) = \delta_{n,k}$. This implies that
$\{K(x_{n},\cdot)\} = \{\phi_{n}\}$ is an orthogonal basis of
$L^{2}(I,d\mu)$ and that $\calP$ has an orthogonal basis given by
$\{\widetilde{\calK}(\chi_I\overline{K(x_{n},\cdot)})\} =
\{\widetilde{\calK}(\chi_I\overline{\phi_{n}(\cdot)})\} =
\{S_{n}\}$. The orthogonal expansion in the basis $\{S_{n}\}_{n\in N}$ 
is the sampling theorem. In \cite{Garcia} there is given a
detailed exposition of a similar structure, which, although restricted
to sampling theory, is in its essence equivalent to the one that we 
have described.
\end{Rem}

The objects that we are interested in here are mainly those expansions that
fit in the above setup, but that are not sampling expansions. As we
will see, there exist quite a few of these. We will see in this work a
wealth of situations where explicit computations of certain integrals yield
new expansion formulas of the type~\eqref{eq:Kxt}, but in general they are special
cases of the more general setting that we will provide in the next section.

\begin{Rem}
\label{rem:unifconv}
Perhaps the most remarkable feature that this setup
inherits from sampling theory is the fact that, in many situations, uniform
convergence can be guaranteed, once we know that the expansion converges in
norm. This happens because $\calP$ is a Hilbert space with a reproducing
kernel given by
\begin{equation*}
  k(x,y) = \sum_{n\in N}S_{n}(x)\overline{S_{n}(y)}
  = \int_{I}K(x,t)\overline{K(y,t)}\,d\mu(t).
\end{equation*}
This fact can be proved using Saitoh's theory of linear transformation
in Hilbert space \cite{Sa, S-book} in a
way similar to what has been done in \cite{Abr} and also by the same
arguments in~\cite{Garcia}. The uniform convergence of the
expansions~\eqref{eq:fcS} is now a consequence of the well-known fact that if 
the sequence $f_{n}$ converges to $f$ in the norm of a Hilbert space with
reproducing kernel $k(\cdot,\cdot)$, then the convergence is pointwise to $f$
and uniform in every set where
$\left\|K(x,\cdot)\right\|_{L^{2}(I,d\mu)}$ is bounded.
\end{Rem}

%------------------------------
\subsection{Biorthogonal expansions}
\label{sub:biorthogonal}
%------------------------------

We now consider the same setup as in subsection~\ref{sub:set-up}
(in particular, the notation for the operators $\calK f$
and its inverse $\widetilde{\calK}g$ in terms of a kernel $K(x,t)$
that satisfies the multiplication formula~\eqref{eq:multiplication}),
but instead of the orthonormal
basis $\{\phi_{n}\}_{n\in N}$ of the space $L^{2}(I,d\mu)$,
we assume that we have a pair of complete biorthonormal
sequences of functions in $L^{2}(I,d\mu)$, $\{P_{n}\}_{n\in N}$
and $\{Q_{n}\}_{n\in N}$, that is,
\begin{equation*}
%\label{eq:biort}
  \int_{I}P_{n}(x)\overline{Q_{m}(x)}\,d\mu(x) = \delta_{n,m}
\end{equation*}
and every $g \in L^2(I,d\mu)$ can be written, in a unique way, as
\begin{equation*}
  g(t) = \sum_{n \in N} c_n(g) P_n(t),
  \qquad
  c_n(g) = \int_I g(t) \overline{Q_n(t)}\,d\mu(t).
\end{equation*}
Let us also define, in $L^2(\Omega,d\mu)$, the sequences of functions $\{S_n\}_{n\in N}$
and $\{T_n\}_{n\in N}$ given by
\begin{equation}
\label{eq:K-Qn}
% \calK(S_{n})(x) = \chi_{I}(x)\overline{Q_{n}(x)}.
  S_{n}(x) = \widetilde{\calK}(\chi_{I}\overline{Q_{n}})(x), \quad x \in \Omega,
\end{equation}
and
\begin{equation*}
  T_{n}(x) = \overline{\calK(\chi_{I}P_{n})(x)}, \quad x \in \Omega
\end{equation*}
(note that if $P_n = Q_n$ then $S_n = T_n$).

Our purpose is to prove the following theorem, which says that it is still
possible to find a bilinear expansion in this context.

\begin{Thm}
\label{thm:expbilin} 
For each $x\in \Omega$, the following expansion\footnote{The 
condition $t \in I$ in the identity~\eqref{eq:expbilin} is \textit{not} a mistake. 
Although $K(x,t)$ is defined on $\Omega \times \Omega$, the functions
$P_n(t)$ are defined, in general, only on~$I$.} holds, with respect to $t$,
in $L^2(I,d\mu)$:
\begin{equation}
\label{eq:expbilin}
  K(x,t) = \sum_{n\in N} P_{n}(t)S_{n}(x).
\end{equation}
Moreover, $\{S_n\}_{n\in N}$ and $\{T_n\}_{n\in N}$ are a pair of complete
biorthogonal sequences in~$\calP$ such that every
$f \in \calP$ can be written as
\begin{equation}
\label{eq:expbilinsamp}
  f(x) = \sum_{n \in N} c_n(f) S_n(x),
  \quad x \in \Omega,
\end{equation}
with
\begin{equation*}
  c_n(f) = \int_{\Omega} f(t) \overline{T_n(t)}\,d\mu(t).
\end{equation*}
The convergence is uniform in every set where $\left\|K(x,\cdot)\right\|_{L^{2}(I,d\mu)}$ is bounded.
\end{Thm}

\begin{proof}
Let us start by proving~\eqref{eq:expbilin}. Since $K(x,\cdot) \in
L^2(I,d\mu)$ for every $x \in \Omega$, as $\{P_n\}_{n \in N}$ and
$\{Q_n\}_{n \in N}$ are a complete biorthogonal system on
$L^2(I,d\mu)$, we can write
\begin{equation*}
  K(x,t) = \sum_{n \in N} b_n(x) P_n(t)
\end{equation*}
with (by~\eqref{eq:K-Qn})
\begin{equation*}
  b_n(x) = \int_{I} K(x,t) \overline{Q_n(t)}\,d\mu(t)
  = \widetilde{\calK}(\chi_I \overline{Q_n})(x)
  = S_n(x).
\end{equation*}

Now, let us prove the biorthogonality of $\{S_n\}_{n\in N}$ and $\{T_n\}_{n\in N}$.
By definition and the multiplicative formula, we have
\begin{align*}
  \int_{\Omega} S_n \overline{T_m}\,d\mu
  &= \int_{\Omega} S_n \calK(\chi_I P_m)\,d\mu \\
  &= \int_{\Omega} \calK(S_n) \chi_I P_m \,d\mu
  = \int_{I} \overline{Q_n} P_m \,d\mu = \delta_{n,m}.
\end{align*}
Finally, for $f \in \calP$, by applying~\eqref{eq:expbilin},
interchanging the sum and the integral, and using the multiplicative
formula, we have
\begin{align*}
  f(x) &= \int_I u(t) K(x,t)\,d\mu(t)
  = \sum_{n \in N} \Big( \int_I u(t) P_n(t)\,d\mu(t) \Big) S_n(x) \\
  &= \sum_{n \in N} \Big( \int_{\Omega} (\calK f)(t) \chi_I(t) P_n(t)\,d\mu(t) \Big) S_n(x) \\
  &= \sum_{n \in N} \Big( \int_{\Omega} f(t) \calK(\chi_I P_n)(t)\,d\mu(t) \Big) S_n(x)
  = \sum_{n \in N} \Big( \int_{\Omega} f(t) \overline{T_n(t)}\,d\mu(t) \Big) S_n(x)
\end{align*}
and the proof is finished.
\end{proof}

%------------------------------
\section{The Fourier kernel}
%------------------------------

As an example to clarify the technique, and to show how useful the use 
of the \textit{biorthogonal} sequences is, let us look at~\eqref{eq:expsamp}
and~\eqref{eq:geg} in the light of the above scheme.

%------------------------------
\subsection{The classical sampling formula}
%------------------------------

With the notation of the above section,
take $d\mu(x) = dx$, $\Omega = \R$, $I = [-1,1]$
and the kernel $K(x,t) = \frac{1}{\sqrt{2\pi}}e^{ixt}$,
so the operator $\calK$ is the Fourier transform.

The space $\calP$ becomes the classical Paley-Wiener space $PW$.
Now, take $N = \Z$ and the functions
\begin{equation*}
  P_{n}(t) = Q_{n}(t) = \phi_{n}(t) = \frac{1}{\sqrt{2}}e^{i\pi nt},
  \quad n \in \Z.
\end{equation*}
Then, $S_{n}(x)$ is
\begin{equation*}
  S_{n}(x) = \widetilde{\calK}(\chi_I\overline{Q_n})(x)
  = \int_{-1}^{1} \frac{e^{ixt}}{\sqrt{2\pi}} \frac{e^{-i\pi nt}}{\sqrt{2}}\,dt
  = \frac{\sin(x-\pi n)}{\sqrt{\pi}(x-\pi n)}.
\end{equation*}
From this expression, by using~\eqref{eq:expbilin} we obtain~\eqref{eq:expsamp}.

Moreover, using the identity
\begin{equation*}
  \frac{1}{\pi} \int_{\R} \frac{\sin(x-\pi n)}{x- \pi n}\,f(x)\,dx
  = f(\pi n),
  \qquad f \in PW,
\end{equation*}
and~\eqref{eq:expbilinsamp}, we deduce the classical
Whittaker-Shannon-Kotel'nikov sampling theorem
\begin{equation*}
  f(x) = \sum_{n=0}^{\infty} \frac{\sin(x- \pi n)}{x-\pi n}\,f(\pi n).
\end{equation*}

%------------------------------
\subsection{Gegenbauer's plane wave expansion}
%------------------------------

As in the previous case,
take $d\mu(x) = dx$, $\Omega = \R$, $I = [-1,1]$,
the kernel $K(x,t) = \frac{1}{\sqrt{2\pi}}e^{ixt}$,
$\calK$ the Fourier transform, and $\calP = PW$.

But, this time, consider $N = \N \cup \{0\}$ and,
using $C_n^\beta(t)$ to denote the Gegenbauer polynomial of order
$\beta > -1/2$ (with the usual trick of employing the Chebyshev
polynomials if $\beta = 0$), take the biorthonormal system
\begin{align*}
  P_{n}(t) &= C_{n}^{\beta}(t),\\
  Q_{n}(t) &= (1-t^2)^{\beta-1/2} C_{n}^{\beta}(t)/h_n
\end{align*}
with
\begin{equation*}
  h_n = \int_{-1}^1 C_{n}^{\beta}(t)^2 (1-t^2)^{\beta-1/2}\,dt
  = \frac{\pi^{1/2}\Gamma(\beta+1/2)\Gamma(2\beta+n)}
  {\Gamma(\beta)\Gamma(2\beta)(n+\beta)n!}.
\end{equation*}
Using the integral
\begin{equation*}
  \int_{\R} e^{-ixt}x^{-\beta}J_{\beta+n}(x)\,dx
  = \frac{2^{-\beta+1} \pi^{1/2} \Gamma(2\beta) (-1)^n i^{n} n!}
    {\Gamma(\beta+1/2)\Gamma(2\beta+n)}
  \,(1-t^2)^{\beta-1/2} C_{n}^{\beta}(t) \chi_{[-1,1]}(t)
\end{equation*}
(see~\cite[Ch.~3.3, (9), p.~123]{E-TIT-I}) we deduce that
\begin{equation*}
  S_{n}(x) = 2^{\beta-1/2} \pi^{-1/2} i^n \Gamma(\beta) (\beta+n)
  x^{-\beta}\,J_{\beta+n}(x).
\end{equation*}
Then, \eqref{eq:expbilin} becomes~\eqref{eq:geg}.

Moreover, every function $f \in PW$ admits an expansion
in a uniformly convergent Fourier-Neumann series of the form
\begin{equation*}
  f(x) = 2^{\beta-1/2} \pi^{-1/2} \Gamma(\beta)
  \sum_{n = 0}^{\infty} c_{n}(f) i^n (\beta+n)
  x^{-\beta}\,J_{\beta+n}(x),
\end{equation*}
with
\begin{equation*}
  c_{n}(f) = \int_{\R} f(t) \calK(\chi_{[-1,1]}C_n^{\beta})(t)\,dt,
\end{equation*}
corresponding to the expansion~\eqref{eq:expbilinsamp}. 
A more explicit expression (see~\eqref{eq:cnf-Dunkl}) for the coefficients $c_n(f)$ 
will be given in the next section for some values of~$\beta$.
Finally, since
\begin{equation*}
  \left\|K(x,\cdot)\right\|_{L^{2}(I,dx)}
  = \frac{1}{\sqrt{2\pi}} \left\|e^{ix\cdot}\right\|_{L^{2}([-1,1],dx)}
  = \frac{1}{\sqrt{\pi}},
\end{equation*}
Remark~\ref{rem:unifconv} automatically ensure that the above mentioned 
expansions converge uniformly on the real line.

%------------------------------
\section{The Dunkl kernel on the real line}
\label{sec:Dunkl}
%------------------------------

%------------------------------
\subsection{The Dunkl transform}
\label{sub:Dunkl}
%------------------------------

For $\alpha > -1$, let $J_\alpha$ denote the Bessel function of order $\alpha$ and,
for complex values of the variable $z$, let
\begin{equation*}
  \calI_\alpha(z)
  = 2^\alpha \Gamma(\alpha+1)\, \frac{J_\alpha(iz)}{(iz)^\alpha}
  = \Gamma(\alpha+1) \sum_{n=0}^\infty \frac{(z/2)^{2n}}{n!\,\Gamma(n+\alpha+1)}
\end{equation*}
($\calI_\alpha$ is a small variation of the so-called modified
Bessel function of the first kind and order~$\alpha$, usually
denoted by $I_\alpha$; see~\cite{Wat}). Moreover, let us take
\begin{equation*}
  E_\alpha(z) = \calI_\alpha(z) + \frac{z}{2(\alpha+1)}\,
    \calI_{\alpha+1}(z),
  \qquad z \in \C.
\end{equation*}

The Dunkl operators on $\R^n$ are differential-difference operators
associated with some finite reflection groups (see~\cite{Dunkl}). We
consider the Dunkl operator $\Lambda_\alpha$, $\alpha \ge -1/2$, associated
with the reflection group $\Z_2$ on $\R$ given by
\begin{equation}
\label{eq:dunkloper}
  \Lambda_\alpha f(x) = \frac{d}{dx}f(x) + \frac{2\alpha+1}{x}
  \left(\frac{f(x)-f(-x)}{2}\right).
\end{equation}
For $\alpha \ge -1/2$ and $\lambda\in \C$, the initial value problem
\begin{equation}
\label{eq:dunklproblem}
\begin{cases}
  \Lambda_\alpha f(x) = \lambda f(x),\quad x\in \R, \\
  f(0) = 1
\end{cases}
\end{equation}
has $E_\alpha(\lambda x)$ as its unique solution (see~\cite{Dunkl2}
and~\cite{Jeu}); this function is called the Dunkl kernel. For $\alpha = -1/2$, it is
clear that $\Lambda_{-1/2} = d/dx$, and $E_{-1/2}(\lambda x) = e^{\lambda x}$.

Let $d\mu_{\alpha}(x) = (2^{\alpha+1} \Gamma(\alpha+1))^{-1}|x|^{2\alpha+1}\,dx$
and write
\begin{equation}
\label{eq:Ea}
  E_{\alpha}(ix) = 2^{\alpha}\Gamma(\alpha+1)
  \left( \frac{J_{\alpha}(x)}{x^{\alpha}} + \frac{J_{\alpha+1}(x)}{x^{\alpha+1}}\,xi\right).
\end{equation}
In a similar way to the Fourier transform (which is the particular
case $\alpha = -1/2$), the Dunkl transform of order $\alpha \geq -1/2$ is
given by
\begin{equation}
\label{eq:Du-T}
  \calF_{\alpha}f(y) = \int_{\R} f(x)E_{\alpha}(-iyx)
  \,d\mu_{\alpha}(x),\quad y\in \R,
\end{equation}
for $f\in L^{1}(\R,d\mu_{\alpha})$. By means of the Schwartz class
$\calS(\R)$, the definition is extended to
$L^{2}(\R,d\mu_{\alpha})$ in the usual way. In \cite{Jeu},
it is shown that $\calF_{\alpha}$ is an isometric isomorphism on
$L^{2}(\R,d\mu_{\alpha})$ and that
\begin{equation*}
  \calF_{\alpha}^{-1}f(y) = \calF_{\alpha}f(-y)
\end{equation*}
for functions such that
$f,\calF_{\alpha}f\in L^{1}(\R,d\mu_{\alpha})$.

From Fubini's theorem, it follows that the Dunkl transform satisfies the multiplication formula
\begin{equation}
\label{eq:mulFa}
  \int_{\R}u(y) \calF_\alpha v(y)\,d\mu_\alpha(y) 
  = \int_{\R} \calF_\alpha u(y) v(y)\,d\mu_\alpha (y).
\end{equation}

Finally, let us take into account that the Dunkl transform
$\calF_{\alpha}$ can also be defined in
$L^2(\R,d\mu_{\alpha})$ for $-1 < \alpha < -1/2$, although the
expression \eqref{eq:Du-T} is no longer valid for
$f\in L^1(\R,d\mu_{\alpha})$ in general.
However, it preserves the same properties in
$L^2(\R,d\mu_{\alpha})$; see~\cite{Rosenblum} for details. This
allows us to extend our study to the case $\alpha > -1$.

%------------------------------
\subsection{The sampling theorem related to the Dunkl transform}
\label{sub:samp-Dunkl}
%------------------------------

In our general setup developed in subsection~\ref{sub:biorthogonal}, let us start by taking
$\Omega = \R$, $I = [-1,1]$, $d\mu = d\mu_\alpha$, with $\alpha>-1$, and
$L^{2}(I,d\mu) = L^{2}([-1,1],d\mu_{\alpha})$.
On this space, we consider the kernel $K(x,t) = E_{\alpha}(ixt)$,
so the corresponding operator is $\calK = \calF_{\alpha}$,
i.e., the abovementioned Dunkl transform.

Now, as usual in sampling theory, we take the space of Paley-Wiener type that,
in our setting, is defined as
\begin{equation}
\label{eq:PWa}
  PW_\alpha = \left\{ f\in L^2(\R,d\mu_\alpha):
  f(x) = \int_{-1}^{1} u(t)E_{\alpha}(ixt)\,d\mu_{\alpha}(t),\;
  u\in L^{2}([-1,1],d\mu_\alpha) \right\}
\end{equation}
endowed with the norm of $L^{2}(\R,d\mu_\alpha)$.
(This space is characterized in \cite[Theorem~5.1]{AJ} as being
the space of entire functions of exponential type $1$ that belong
to $L^2(\R,d\mu_\alpha)$ when restricted to the real line.) Then,
take, of course $\calP = PW_{\alpha}$.

It is well-known that the Bessel function $J_{\alpha+1}(x)$ has an
increasing sequence of positive zeros $\{s_n\}_{n\ge1}$.
Consequently, the real function
$\operatorname{Im}(E_{\alpha}(ix)) = \frac{x}{2(\alpha+1)}\,\calI_{\alpha+1}(ix)$
is odd and it has an infinite sequence of zeros $\{s_n\}_{n\in \Z}$
(with $s_{-n} = -s_n$ and $s_0 = 0$).
Then, following~\cite{CV} (or~\cite{CPRV}), let us define the functions
\begin{equation}
\label{eq:ean}
  e_{\alpha,n}(t) = d_n E_{\alpha}(is_n t),
  \quad n\in \Z, \quad t \in [-1,1],
\end{equation}
where
\begin{equation*}
  d_n = \frac{2^{\alpha/2}(\Gamma(\alpha+1))^{1/2}}
    {|\calI_{\alpha}(is_n)|}, \quad n \ne 0,
  \qquad
  d_0 = 2^{(\alpha+1)/2}(\Gamma(\alpha+2))^{1/2}.
\end{equation*}
With this notation, the sequence of functions
$\{e_{\alpha,n}\}_{n\in \Z}$
is a complete orthonormal system in $L^2([-1,1],d\mu_{\alpha})$, for $\alpha>-1$.
Thus, let us take $N = \Z$ and
$P_{n}(t) = Q_n(t) = e_{\alpha,n}(t)$.

On the other hand, let us use that, for $x,y\in \R$, $x \ne  y$ and $\alpha>-1$, we have
\begin{equation}
\label{eq:Lommel-Dunkl}
  \int_{-1}^1 E_{\alpha}(ixt)\overline{E_{\alpha}(iyt)}\, d\mu_{\alpha}(t)
  = \frac{1}{2^{\alpha+1}\Gamma(\alpha+2)}
    \frac{x\calI_{\alpha+1}(ix)\calI_{\alpha}(iy)
      - y\calI_{\alpha+1}(iy)\calI_{\alpha}(ix)}{x-y}
\end{equation}
(the proof can be found in~\cite{BCV} or~\cite{CV}).
Then,
\begin{align*}
  S_{n}(x) &= \widetilde{\calK}(\chi_{[-1,1]}\overline{Q_n})(x)
  = \int_{-1}^1 E_\alpha (ixt) \overline{e_{\alpha,n}(t)} \,d\mu_\alpha(t) \\
  &= \frac{d_n}{2^{\alpha+1}\Gamma(\alpha+2)}
  \frac{x \calI_{\alpha+1}(ix) \calI_{\alpha}(is_n)
    - s_n \calI_{\alpha+1}(is_n) \calI_{\alpha}(ix) }{x-s_n} \\
  &= \frac{d_n}{2^{\alpha+1}\Gamma(\alpha+2)}
  \frac{x \calI_{\alpha+1}(ix) \calI_{\alpha}(is_n)}{x-s_n}
\end{align*}
because $\calI_{\alpha+1}(is_n) = 0$.
Consequently,
\begin{align*}
  E_\alpha(ixt) &= \sum_{n \in \Z} e_{\alpha,n}(t)
  \frac{d_n}{2^{\alpha+1}\Gamma(\alpha+2)}
  \frac{x \calI_{\alpha+1}(ix) \calI_{\alpha}(is_n)}{x-s_n} \\
  &= \calI_{\alpha+1}(ix) + \sum_{n\in\Z \setminus\{0\}}
  \frac{E_{\alpha}(is_nt)}{2(\alpha+1) \calI_{\alpha}(is_n)}
  \frac{x \calI_{\alpha+1}(ix)}{x-s_n},
\end{align*}
which corresponds to the formula~\eqref{eq:expbilin} in Theorem~\ref{thm:expbilin}.

Finally, the formula~\eqref{eq:expbilinsamp} in Theorem~\ref{thm:expbilin} says that,
if $f\in PW_\alpha$, then $f$ has the representation
\begin{equation}
\label{eq:samp-Dunkl}
  f(x) = f(s_0) \calI_{\alpha+1}(ix)
  + \sum_{n\in\Z \setminus\{0\}} f(s_n)
    \frac{x\calI_{\alpha+1}(ix)}{2(\alpha+1)\calI_{\alpha}(is_n)(x-s_n)},
\end{equation}
that converges in the norm of $L^{2}(\R,d\mu_\alpha)$. This is so because
the coefficients $c_n(f)$ in~\eqref{eq:expbilinsamp} are $c_n(f) = d_n f(s_n)$,
as we can see in what follows:
\begin{align*}
  c_n(f) &= \int_{\R} f(t) \calF_\alpha (\chi_{[-1,1]}e_{\alpha,n})(t) \,d\mu_\alpha(t) \\
  &= \int_{\R} f(t) \int_{-1}^1 e_{\alpha,n}(x) E_\alpha(-ixt) \,d\mu_\alpha(x) \,d\mu_\alpha(t) \\
  &= \int_{-1}^1 e_{\alpha,n}(x) \int_{\R} f(t) E_\alpha(-ixt) \,d\mu_\alpha(t) \,d\mu_\alpha(x) \\
  &= d_n \int_{-1}^1 E_{\alpha}(is_nx) \int_{\R} f(t) E_\alpha(-ixt) \,d\mu_\alpha(t) \,d\mu_\alpha(x)
  = d_n f(s_n),
\end{align*}
where in the last step we have used that $f\in PW_\alpha$.
%\begin{equation*}
%  \frac{1}{2^{\alpha+1}\Gamma(\alpha+2)}
%  \int_{\R} \frac{t\calI_{\alpha+1}(it)\calI_\alpha(is_n)}{t-s_n}f(t)\,d\mu_\alpha(t) = f(s_n).
%\end{equation*}
Moreover, by using L'Hopital rule in \eqref{eq:Lommel-Dunkl}, it is not difficult to check that
\begin{align*}
\qquad
\left\| \frac{E_\alpha(x \cdot)}{(x \cdot)^\alpha} \right\|^2_{L^2([-1,1],d\mu_\alpha)}
  &= \int_{-1}^1 |E_{\alpha}(ixr)|^2\, d\mu_{\alpha}(r) \\
  &= \frac{1}{2^{\alpha+1}\Gamma(\alpha+2)}
  \bigg( \frac{x^2}{2(\alpha+1)}\calI_{\alpha+1}^2(ix) \\
  &\qquad- (2\alpha+1)\calI_{\alpha+1}(ix)\calI_{\alpha}(ix)
  + 2(\alpha+1)\calI_{\alpha}^2(ix) \bigg),
\end{align*}
and this norm is bounded on every compact set on the real line.
So, by Remark~\ref{rem:unifconv}, the series~\eqref{eq:samp-Dunkl} converges uniformly in compact subsets of~$\R$.
\eqref{eq:samp-Dunkl} is the sampling theorem related to the Dunkl transform that
has been established in~\cite{CV}.

%------------------------------
\subsection{Fourier-Neumann type expansion}
\label{sub:main-Dunkl}
%------------------------------

Following~\cite[Definition~1.5.5, p.~27]{DX}, let us introduce the generalized
Gegenbauer polynomials $C_n^{(\lambda,\nu)}(t)$ for $\lambda>-1/2$, $\nu \ge 0$
and $n \ge 0$
(the case $\nu=0$ corresponding to the ordinary Gegenbauer polynomials);
actually, for convenience with the notation of this paper,
we are going to use $C_n^{(\beta+1/2,\alpha+1/2)}(x)$. In this way, for $\beta>-1$
and $\alpha\ge -1/2$, the generalized Gegenbauer polynomials are defined by
\begin{align}
\label{eq:C2n}
C_{2n}^{(\beta+1/2,\alpha+1/2)}(t)
&= (-1)^n\,\frac{(\alpha+\beta+1)_n}{(\alpha+1)_n} P_n^{(\alpha,\beta)}(1-2t^2), \\
\label{eq:C2n1}
C_{2n+1}^{(\beta+1/2,\alpha+1/2)}(t)
&= (-1)^n\,\frac{(\alpha+\beta+1)_{n+1}}{(\alpha+1)_{n+1}} t P_n^{(\alpha+1,\beta)}(1-2t^2),
\end{align}
where in the coefficients we are using the Pochhammer symbol
$(a)_n = a(a+1)\cdots(a+n-1) = \Gamma(a+n)/\Gamma(a)$. Note that there is no problem in
extending the definition of the generalized Gegenbauer polynomials taking $\alpha>-1$, 
so we will assume this situation.

From the $L^2$-norm of the Jacobi polynomials (see~\cite[Ch.~16.4,
(5), p.~285]{E-TIT-II}), it is easy to find
\begin{align}
\label{eq:hC2n}
  h_{2n}^{(\beta,\alpha)} &= \int_{-1}^1 \left[C_{2n}^{(\beta+1/2,\alpha+1/2)}(t)\right]^2
  (1-t^2)^\beta \,d\mu_\alpha(t) \\
  &= \frac{1}{2^{\alpha+1}} 
  \frac{\Gamma(\alpha+1)\Gamma(\beta+n+1)\Gamma(\alpha+\beta+n+1)}{(\alpha+\beta+2n+1)
  \Gamma(\alpha+\beta+1)^2 \Gamma(\alpha+n+1)\, n!},
\notag \\
\label{eq:hC2n1}
  h_{2n+1}^{(\beta,\alpha)} &= \int_{-1}^1 \left[C_{2n+1}^{(\beta+1/2,\alpha+1/2)}(t)\right]^2
  (1-t^2)^\beta \,d\mu_\alpha(t) \\
  &= \frac{1}{2^{\alpha+1}} \frac{\Gamma(\alpha+1)\Gamma(\beta+n+1)\Gamma(\alpha+\beta+n+2)}
  {(\alpha+\beta+2n+2) \Gamma(\alpha+\beta+1)^2 \Gamma(\alpha+n+2)\, n!}.
\notag
\end{align}

Finally, given $\alpha>-1$, we define the functions $\calJ_{\alpha,n}(x)$ by 
\begin{equation*}
  \calJ_{\alpha,n}(x) = \frac{J_{\alpha+n+1}(x)}{x^{\alpha+1}},
  \qquad x \in \R,
  \quad n = 0,1,2,\dots;
\end{equation*}
as these functions arise in Fourier-Neumann series, we will allude to 
$\calJ_{\alpha,n}(x)$ using the name of Neumann functions.\footnote{In the literature, 
the name ``Neumann functions'' is sometimes used for the Bessel functions of the 
second kind $Y_\alpha(x)$, but these functions will not arise in this paper.}
From the identities
\begin{gather*}
  \int_0^{\infty} \frac{J_a(x) J_b(x)}{x}\,dx
  = \frac{2}{\pi} \frac{\sin((b-a)\pi/2)}{b^2-a^2},
  \quad a>0,\ b>0,\ a \ne b,
  \\
  \int_0^{\infty} \frac{J_a(x)^2}{x}\,dx
  = \frac{1}{2a}, \quad a > 0,
\end{gather*}
and taking into account that $\calJ_{\alpha,n}(x)$ is even or odd
according to $n$ being even or odd, respectively, it follows that
$\{\calJ_{\alpha,n}\}_{n\ge0}$ is an orthogonal system on
$L^2(\R,d\mu_\alpha)$, namely,
\begin{equation*}
  \int_{\R} \calJ_{\alpha,n}(x) \calJ_{\alpha,m}(x)
  \, d\mu_{\alpha}(x) = \frac{\delta_{n,m}}{2^{\alpha+1}\Gamma(\alpha+1)(\alpha+n+1)}.
\end{equation*}

Generalized Gegenbauer polynomials and Neumann functions are the
main ingredients for obtaining the Dunkl analogue of Gegenbauer's expansion of the plane wave. 
To establish this result we need a relation between them. By using the notation
\begin{align}
\label{eq:calP}
  \calP^{(\alpha,\beta)}_n(t)
    &= C_n^{(\beta+1/2,\alpha+1/2)}(t),\\
\label{eq:calQ}
  \calQ^{(\alpha,\beta)}_n(t)
    &= \big(h_n^{(\beta,\alpha)}\big)^{-1} (1-t^2)^\beta C_n^{(\beta+1/2,\alpha+1/2)}(t)
\end{align}
(where $h_n^{(\beta,\alpha)}$ is given in~\eqref{eq:hC2n} and~\eqref{eq:hC2n1}),
this relation is given in the following lemma that, moreover,
can have independent interest:

\begin{Lem}
\label{lem:F-PQ}
Let $\alpha,\beta > -1$, $\alpha+\beta>-1$,
and $k = 0,1,2,\dots$. The Dunkl transform of order $\alpha$ of
$\calJ_{\alpha+\beta,k}(x)$ is
\begin{equation}
\label{eq:F-Q}
  \calF_{\alpha}(\calJ_{\alpha+\beta,k})(t)
  % = \text{constant} (1-t^2)^\beta
  % \mathcal{C}_{n}^{(\beta+1/2,\beta+1/2)}(t) \chi_{[-1,1]}(t).
  = \frac{(-i)^k}{2^{\alpha+\beta+1}\Gamma(\alpha+\beta+1)(\alpha+\beta+k+1)}
  \, \calQ^{(\alpha,\beta)}_{k}(t)\chi_{[-1,1]}(t).
\end{equation}
Moreover, if $\beta<1$, we have\footnote{Observe that nothing is
said about outside the interval $[-1,1]$; this does not mean that this
expression vanishes for $|t|>1$, that is not true when 
$\beta \ne 0$.}
\begin{equation}
\label{eq:F-P}
  \calF_{\alpha}(|\cdot|^{2\beta}\calJ_{\alpha+\beta,k})(t)
  % = \textrm{constant} \mathcal{C}_{k}^{(\beta+1/2,\alpha+1/2)}(t).
  = \frac{2^\beta \Gamma(\alpha+\beta+1)}{\Gamma(\alpha+1)}
  \, (-i)^k \calP^{(\alpha,\beta)}_{k}(t),\quad t\in [-1,1].
\end{equation}
\end{Lem}

\begin{Rem}
\label{rem:Ntoinfty}
There is a delicacy with the formulas in Lemma~\ref{lem:F-PQ}. 
Actually, $\calF_{\alpha}$ was defined, as a
first step, as a Lebesgue integral for suitably integrable functions. Then,
$\calF_{\alpha}$ is extended to $L^p$ spaces where the integral representation is no
longer valid for some functions. Now, the integrals 
\begin{equation*}
  \int_{0}^{\infty}x^{-\lambda }J_{\mu }(ax)J_{\nu }(bx)\,dx
\end{equation*}
from~\cite[Ch.~8.11]{E-TIT-II} or~\cite[Ch.~XIII]{Wat}
that we will use in the proof are improper Riemann integrals. Hence, the proper understanding
of those integrals should be as
\begin{equation*}
  \lim_{N\to\infty}
  \int_{0}^{N} x^{-\lambda }J_{\mu }(ax)J_{\nu }(bx)\,dx.
\end{equation*}
Since $\calJ_{\alpha+\beta,k} \chi_{[-N,N]}$ and 
$|\cdot|^{2\beta}\calJ_{\alpha+\beta,k} \chi_{[-N,N]}$
are integrable functions, the integral form
of $\calF_{\alpha}$ is valid here and \eqref{eq:F-Q} and \eqref{eq:F-P}
can be understood as
\begin{equation*}
  \lim_{N\to\infty} \calF_{\alpha}(\calJ_{\alpha+\beta,k} \chi_{[-N,N]})(t),
  \qquad
  \lim_{N\to\infty} \calF_{\alpha}(|\cdot|^{2\beta}\calJ_{\alpha+\beta,k} \chi_{[-N,N]})(t),
\end{equation*}
and the identities in the lemma hold in the almost everywhere sense. 
Finally, the $L^2$ boundedness of $\calF_{\alpha}$ allows us to understand
these identities in the $L^2$ sense.
From now on, we will no longer mention these details.
\end{Rem}

We postpone the proof of Lemma~\ref{lem:F-PQ} to subsection~\ref{sub:prooflemma}. 
With Lemma~\ref{lem:F-PQ}, we already have all the tools for proving:

\begin{Thm}
\label{thm:Dunkl}
Let $\alpha,\beta > -1$ and $\alpha+\beta>-1$. Then for each $x\in \R$ the following expansion holds in
$L^{2}([-1,1],d\mu_{\alpha})$:
\begin{equation}
\label{eq:JA-Dunkl}
  E_\alpha(ixt) = 2^{\alpha+\beta+1} \Gamma(\alpha+\beta+1)
  \sum_{n = 0}^{\infty} i^{n}(\alpha+\beta+n+1)
  \calJ_{\alpha+\beta,n}(x) C_{n}^{(\beta+1/2,\alpha+1/2)}(t).
\end{equation}
Moreover, for $\beta < 1$ and $f \in PW_\alpha$, we have the orthogonal expansion
\begin{equation}
\label{eq:f-Dunkl}
  f(x) = \sum_{n=0}^\infty a_n(f) (\alpha+\beta+n+1) \calJ_{\alpha+\beta,n}(x)
\end{equation}
with
\begin{equation}
\label{eq:cnf-Dunkl}
  a_n(f) = 2^{\alpha+\beta+1} \Gamma(\alpha+\beta+1)
  \int_{\R} f(t) \calJ_{\alpha+\beta,n}(t) \,d\mu_{\alpha+\beta}(t).
\end{equation}
Furthermore, the series converges uniformly in compact subsets of~$\R$.
\end{Thm}

\begin{proof}
In the biorthogonal setup given in subsection~\ref{sub:biorthogonal}, let 
$\Omega = \R$, $I = [-1,1]$,
the space $L^{2}(I,d\mu) = L^{2}([-1,1],d\mu_{\alpha})$,
and the kernel $K(x,t) = E_{\alpha}(ixt)$,
from which the operator $\calK$ becomes the Dunkl transform~$\calF_{\alpha}$
(and $\widetilde{K} = \calF_{\alpha}^{-1}$).
Also, consider the Paley-Wiener space $\calP = PW_{\alpha}$ (see~\eqref{eq:PWa}).
Finally, for $N= \N \cup \{0\}$, take the biorthonormal system given by
$P_{n}(t) = \calP_{n}^{(\alpha,\beta)}(t)$
and
$Q_{n}(t) = \calQ_{n}^{(\alpha,\beta)}(t)$ as in~\eqref{eq:calP}
and~\eqref{eq:calQ}.
From \eqref{eq:F-Q}, we have
\begin{equation*}
  S_{n}(x) = 2^{\alpha+\beta+1} \Gamma(\alpha+\beta+1) i^n
  (\alpha+\beta+n+1) \calJ_{\alpha+\beta,n}(x).
\end{equation*}
In this situation, the formula~\eqref{eq:expbilin} in Theorem~\ref{thm:expbilin}
gives~\eqref{eq:JA-Dunkl}.

Now, let us consider
\begin{equation*}
  T_n(x) = \overline{\calK(\chi_{[-1,1]}P_n)(x)}
  = \overline{\calF_{\alpha}(\chi_{[-1,1]} \calP^{(\alpha,\beta)}_n)(x)}.
\end{equation*}
Then, the identity given by \eqref{eq:expbilinsamp} becomes
\begin{equation*}
  f(x) = 2^{\alpha+\beta+1} \Gamma(\alpha+\beta+1)
  \sum_{n=0}^\infty c_n(f) i^n (\alpha+\beta+n+1)
  \calJ_{\alpha+\beta,n}(x),
  \quad f\in PW_{\alpha},
\end{equation*}
with
\begin{equation*}
  c_n(f) = \int_{\R} f(t) \calF_{\alpha}(\chi_{[-1,1]}
    \calP^{(\alpha,\beta)}_n)(t) \,d\mu_\alpha(t).
\end{equation*}
Let us see that, when $\beta<1$, the coefficient $c_n(f)$ can be written as
\begin{equation*}
  c_n(f) = (-i)^n \int_{\R} f(t) \calJ_{\alpha+\beta,n}(t)
  \,d\mu_{\alpha+\beta}(t),
\end{equation*}
which implies~\eqref{eq:cnf-Dunkl} with
$a_n(f) = 2^{\alpha+\beta+1} \Gamma(\alpha+\beta+1) i^n c_n(f)$.

Indeed, if we consider $u$ such that
$f = \calF_{\alpha}^{-1}(u\chi_{[-1,1]})$
and use the multiplication formula~\eqref{eq:mulFa}, we can write
\begin{equation*}
  c_n(f) = \int_{-1}^1 u(x) \calP^{(\alpha,\beta)}_n(x) \,d\mu_\alpha(x).
\end{equation*}
Now, by~\eqref{eq:F-P} and the multiplication formula again, we have
\begin{align*}
  c_n(f) &= \int_{-1}^1 u(x) \frac{i^n \Gamma(\alpha+1)}{2^\beta \Gamma(\alpha+\beta+1)}
  \calF_\alpha(|\cdot|^{2\beta} \calJ_{\alpha+\beta,n})(x) \,d\mu_\alpha(x) \\
  &= i^n \int_{\R} \calF_{\alpha}(u\chi_{[-1,1]})(t)
  \calJ_{\alpha+\beta,n}(t) \,d\mu_{\alpha+\beta}(t).
\end{align*}
It is clear that $\calF_{\alpha}(u\chi_{[-1,1]})(t) = f(-t)$, 
so the change of variable from $t$ to $-t$ gives
\begin{equation*}
  c_n(f) = (-i)^n \int_{\R} f(t) \calJ_{\alpha+\beta,n}(t)
  \,d\mu_{\alpha+\beta}(t)
\end{equation*}
because $\calJ_{\alpha+\beta,n}(-t) = (-1)^n \calJ_{\alpha+\beta,n}(t)$.
\end{proof}

% Thus, from the point of view of the novelty of the formula~\eqref{eq:JA-Dunkl}, ...

\begin{Rem}
\label{rem:Rosler}
Actually, formulas \eqref{eq:geg} and \eqref{eq:JA-Dunkl} are equivalent for $\alpha\ge -1/2$. The proof in one direction is clear, just by specializing the parameters. To obtain \eqref{eq:JA-Dunkl} from \eqref{eq:geg}, we can use the intertwining operator
\begin{equation*}
  V_\alpha g(t) = \frac{\Gamma(2\alpha+2)}{2^{2\alpha+1} \Gamma(\alpha+1/2)\Gamma(\alpha+3/2)}
  \int_{-1}^1 g(st) (1-s)^{\alpha-1/2} (1+s)^{\alpha+1/2}\,ds
\end{equation*}
(see~\cite[Definition~1.5.1, p.~24]{DX}, we change the parameter $\mu$ in the definition given in \cite{DX} by $\alpha+1/2$), defined for $\alpha\ge -1/2$. With this notation we have
\begin{equation}
\label{eq:VaCn}
  V_\alpha C_n^{\alpha+\beta+1}(t) = C_n^{(\beta+1/2,\alpha+1/2)}(t)
\end{equation}
and
\begin{equation*}
  V_\alpha e^{i\cdot}(t) = E_\alpha(it).
\end{equation*}
In this way, applying $V_\alpha$ to~\eqref{eq:geg} (with $\alpha+\beta+1$ instead of~$\beta$) we get~\eqref{eq:JA-Dunkl}.
This idea has been used for the higher rank in~\cite{Rosler}, where the author assumes that \eqref{eq:geg} was already known, and then \eqref{eq:JA-Dunkl} is established for $\alpha\ge -1/2$ by using the intertwining operator. This gives a considerably shorter proof. Instead, with the method followed in the proof of Theorem~\ref{thm:Dunkl}, the identity \eqref{eq:JA-Dunkl} not only can be found for $\alpha>-1$, but also is proved directly and then, as a particular case, \eqref{eq:geg} holds.
\end{Rem}
% With the method followed in the proof of Theorem~\ref{thm:Dunkl}, we do not only 
% find \eqref{eq:JA-Dunkl} for $\alpha>-1$, but we also prove \eqref{eq:JA-Dunkl} 
% directly and then, as a particular case, \eqref{eq:geg} holds.

\begin{Rem}
Another way of obtaining~\eqref{eq:JA-Dunkl} is as follows. 
Some results from~\cite{FW} were generalized in~\cite{Ve} to
%\begin{multline*}
\begin{equation*}
  \sum_{m=0}^{\infty} a_mb_m \frac{(zw)^m}{m!} 
  = \sum_{n=0}^{\infty} \frac{(-z)^n}{n!\, (\gamma+n)_n} 
  \left( \sum_{r=0}^\infty \frac{b_{n+r}z^r}{r!\,(\gamma+2n+1)_r} \right) 
%  \\ \times 
  \left( \sum_{s=0}^n \frac{(-n)_s (n+\gamma)_s}{s!} \,a_s w^s \right).
\end{equation*}
%\end{multline*}
When $z$ and $w$ are replaced by $z\gamma$ and $w/\gamma$, respectively, 
and we let $\gamma \to \infty$, we get the
companion formula
\begin{equation*}
  \sum_{m=0}^{\infty} a_mb_m (zw)^m = \sum_{n=0}^{\infty} \frac{(-z)^n}{n!}
  \Biggl( \sum_{j=0}^\infty \frac{b_{n+j}}{j!} \, z^j \Biggr) 
  \Biggl( \sum_{k=0}^n (-n)_k a_k w^k \Biggr)
\end{equation*}
(these formulas are also stated in~\cite[Ch.~9]{Ism}).
These expansions contain several expansions in terms of Jacobi polynomials 
of argument $1 - 2t^2$ (i.e., generalized Gegenbauer polynomials). 
In particular, \eqref{eq:JA-Dunkl} follows in this way.
\end{Rem}

%------------------------------
\subsection{Consequences for the Hankel transform}
%------------------------------

For $\alpha>-1$, consider the so-called modified Hankel transform $H_{\alpha}$, that is
\begin{equation}
\label{eq:htr}
  H_{\alpha}f(y) = \int_{0}^{\infty} \frac{J_{\alpha}(xy)}{(xy)^{\alpha}}
  \,f(y)x^{2\alpha+1}\,dx, \quad x>0.
\end{equation}
The kernel $E_\alpha(ixt)$ of the Dunkl transform~\eqref{eq:Du-T}
can be written in terms of the Bessel functions of order $\alpha$ and
$\alpha+1$, and this clearly allows us to study the Hankel transform
as a simple consequence of the Dunkl transform. In particular, if
we have a function $f \in L^2((0,\infty),x^{2\alpha+1}\,dx)$, we
can take the even extension $f(|\cdot|) \in L^2(\R,d\mu_\alpha)$.
Then, using that $J_{\alpha}(x)/x^\alpha$ is even and
$J_{\alpha+1}(x)/x^\alpha$ is odd, we write~\eqref{eq:htr} as
\begin{equation*}
  H_{\alpha}f(y)
  = \calF_{\alpha}(f(|\cdot|))(y).
\end{equation*}
The Paley-Wiener space for the Hankel transform is given by
\begin{multline*}
  \quad
  PW'_{\alpha} = \Bigg\{ f\in L^{2}((0,\infty),x^{2\alpha+1}\,dx)
  : f(t) = \int_{0}^{1} u(x) \frac{J_{\alpha}(xt)}{(xt)^\alpha}\,x^{2\alpha+1}\,dx,\\
  \; u\in L^{2}((0,1),x^{2\alpha+1}\,dx) \Bigg\};
  \quad
\end{multline*}
also, note that if $f \in PW'_{\alpha}$, then both the even
extension $f(|\cdot|)$ and the odd extension
$\operatorname{sgn}(\cdot)f(|\cdot|)$ belong to~$PW_{\alpha}$.

So, let us adapt the sampling formula of
subsection~\ref{sub:samp-Dunkl} and the Theorem~\ref{thm:Dunkl} of
subsection~\ref{sub:main-Dunkl} to the context of the Hankel
transform.

%----------------------------------
\subsubsection{The sampling formula for the Hankel transform.}
%----------------------------------

For $f \in PW'_{\alpha}$, taking its even extension $f(|\cdot|)$,
using that $s_{-n} = -s_n$ and grouping the summands corresponding
to $1/(x-s_n)$ and $1/(x+s_n)$ in~\eqref{eq:samp-Dunkl}, we get
\begin{equation*}
  f(x) = f(s_0) \calI_{\alpha+1}(ix)
  + \sum_{n=1}^{\infty} f(s_n)
    \frac{\calI_{\alpha+1}(ix)}{(\alpha+1)\calI_{\alpha}(is_n)}
    \frac{x^2}{x^2-s_n^2}.
\end{equation*}
Similarly, with the odd extension of $f$, \eqref{eq:samp-Dunkl} becomes
\begin{equation*}
  f(x) = \sum_{n=1}^{\infty} f(s_n)
    \frac{\calI_{\alpha+1}(ix)}{(\alpha+1)\calI_{\alpha}(is_n)}
    \frac{s_n^2}{x^2-s_n^2}.
\end{equation*}
The latter identity corresponds to the well-known Higgins sampling theorem for the Hankel transform~\cite{Higgins}.

%----------------------------------
\subsubsection{A version of the Theorem~\ref{thm:Dunkl} for the Hankel transform.}
%----------------------------------

Let us observe that
\begin{equation*}
  \frac{J_{\alpha}(xt)}{(xt)^{\alpha}}
  = \frac{1}{2^{\alpha+1}\Gamma(\alpha+1)}
  \bigl( E_\alpha(ixt) + \overline{E_\alpha(ixt)} \bigr).
\end{equation*}
From this, it is very easy to adapt~\eqref{eq:JA-Dunkl} to the new
context, and to write it in terms of Jacobi polynomials by
using~\eqref{eq:C2n}. Given $f \in PW'_\alpha$, let us consider its
even extension $f(|\cdot|) \in PW_\alpha$.
Applying~\eqref{eq:f-Dunkl} to this even function, it becomes an
expansion that only contains $\calJ_{\alpha+\beta,2n}(x) =
J_{\alpha+\beta+2n+1}(x)/x^{\alpha+\beta+1}$ (i.e., only with even
indexes).

Thus, the results corresponding to the Hankel transforms can be summarized in this way:

\begin{Cor}
Let $\alpha,\beta > -1$ and $\alpha+\beta>-1$. Then for each $x\in (0,\infty)$ 
the following expansion holds in
$L^{2}((0,1),x^{2\alpha+1}\,dx)$:
\begin{equation}
\label{eq:Jxta}
  \frac{J_{\alpha}(xt)}{(xt)^\alpha}
  = \sum_{n = 0}^{\infty}
  \frac{2^{\beta+1}(\alpha+\beta+2n+1) \Gamma(\alpha+\beta+n+1)} {\Gamma(\alpha+n+1)}
  \calJ_{\alpha+\beta,2n}(x) P_{n}^{(\alpha,\beta)}(1-2t^{2}).
\end{equation}
Moreover, for $\beta < 1$ and $f \in PW'_\alpha$, we have the orthogonal expansion
\begin{equation*}
  f(x) = \sum_{n=0}^\infty a_n(f) (\alpha+\beta+2n+1)
  \calJ_{\alpha+\beta,2n}(x)
\end{equation*}
with
\begin{equation*}
  a_n(f) = 2 \int_{0}^{\infty} f(t) \calJ_{\alpha+\beta,2n}(t)
  \,t^{2\alpha+2\beta+1}\,dt.
\end{equation*}
Furthermore, the series converges uniformly in compact subsets of~$(0,\infty)$.
\end{Cor}

\begin{Rem}
\label{rem:referee}
In the particular case $\alpha=-1/2$, on using 
$J_{-1/2}(z) = 2^{1/2} \pi^{-1/2} z^{-1/2} \cos(z)$ 
and~\eqref{eq:C2n}, the formula~\eqref{eq:Jxta} becomes
\begin{equation}
\label{eq:cosxt}
  \frac{x^{\beta+1/2}}{2^{\beta+1/2}\Gamma(\beta+1/2)} \cos(xt) 
  = \sum_{n = 0}^{\infty} (-1)^n (2n+\beta+1/2) J_{\beta+2n+1/2}(x)
  C_{2n}^{\beta+1/2}(t),
\end{equation}
which is already known (see \cite[\S\,11.5, formula~(5), p.~369]{Wat}). 
On the other hand, following the procedure described in Remark~\ref{rem:Rosler}, 
from this expression we can obtain another proof of~\eqref{eq:Jxta}, 
valid for $\alpha > -1/2$. Let us assume that~\eqref{eq:cosxt} is already known,
and we write it with $\alpha+\beta+1/2$ instead of~$\beta$;
then, by applying the intertwining operator $V_\alpha$ and using~\eqref{eq:VaCn}, 
\eqref{eq:C2n} and
\begin{equation*}
  V_\alpha \cos(t) = 2^\alpha \Gamma(\alpha+1) \frac{J_\alpha(t)}{t^\alpha},
\end{equation*}
we get the identity~\eqref{eq:Jxta}, as desired.
\end{Rem}

Let us conclude by observing that the $L^p$ convergence of the
orthogonal series that appear in the previous corollary has been
studied in the papers~\cite{V, CGPV} for functions in an appropriate
$L^p$ extension of the Paley-Wiener space.

%------------------------------
\section{Technical lemmas}
%------------------------------

The main goal of this section is to prove Lemma~\ref{lem:F-PQ},
which is key in our study of the Dunkl
transform on the real line. The proof is contained in 
subsection~\ref{sub:prooflemma}. With this target, we need 
to previously establish some formulas. They are given in 
subsection~\ref{sub:intBessel}.

%------------------------------
\subsection{Some integrals involving Bessel functions}
\label{sub:intBessel}
%------------------------------

For the sake of completeness, let us start by proving two identities
that express some integrals involving the product of two Bessel
functions in terms of Jacobi polynomials. Such integrals
are usually written in terms of hypergeometric functions; 
however their expressions as Jacobi polynomials are not
easily found in the literature. For instance, they do not appear
in the standard references~\cite{E-HTF-II, Wat, E-TIT-II}.
In what follows, we can take into account the comments in 
Remark~\ref{rem:Ntoinfty}, but we will not repeat them.

\begin{Lem}
\label{lem:HTjnab}
For $\alpha,\beta>-1$ with $\alpha+ \beta> -1$,
and $n = 0,1,2,\dots$, let us define
\begin{align*}
  I_{-}(\alpha,\beta,n)(t) &= t^{-\alpha}
    \int_0^\infty x^{-\beta} J_{\alpha+\beta+2n+1}(x) J_\alpha(xt)\, dx, \\
  I_{+}(\alpha,\beta,n)(t) &= t^{-\alpha}
    \int_0^\infty x^{\beta} J_{\alpha+\beta+2n+1}(x) J_\alpha(xt)\, dx.
\end{align*}
Then, we have
\begin{equation}
\label{eq:HTjnab}
  I_{-}(\alpha,\beta,n)(t)
  = 2^{-\beta} \tfrac{\Gamma(n+1)}{\Gamma(\beta+n+1)} (1-t^2)^\beta
  P^{(\alpha,\beta)}_n(1-2t^2) \chi_{[0,1]}(t), \quad t \in (0,\infty).
\end{equation}
Assume further that $\beta<1$; then,
\begin{equation}
\label{eq:HTbjnab}
  I_{+}(\alpha,\beta,n)(t)
  = 2^\beta\, \tfrac{\Gamma(\alpha+\beta+n+1)} {\Gamma(\alpha+n+1)}
  P^{(\alpha,\beta)}_n(1-2t^2), \quad t \in (0,1).
\end{equation}
\end{Lem}

\begin{proof}
We use the formula
\begin{multline}
\label{eq:hforig}
  \int_{0}^{\infty}x^{-\lambda }J_{\mu }(ax)J_{\nu }(bx)\,dx \\
  = \frac{b^{\nu }a^{\lambda -\nu -1}\Gamma(\frac{\mu +\nu -\lambda+1}{2})}
  {2^{\lambda }\Gamma(\nu +1)\Gamma(\frac{\lambda+\mu -\nu +1}{2})}
  {\,{}_{2}F_{1}} \left( \frac{\mu +\nu -\lambda+1}{2},
  \frac{\nu-\lambda-\mu+1}{2}; \nu+1; \frac{b^{2}}{a^{2}}\right),
\end{multline}
valid when $0<b<a$ and $-1<\lambda<\mu+\nu+1$;
here, ${{}_{2}F_{1}}$ denotes the hypergeometric function
(see~\cite[Ch.~8.11, (9), p.~48]{E-TIT-II} or \cite[Ch.~XIII, 13.4
(2), p.~401]{Wat}).

Then, let us start with~\eqref{eq:HTjnab}. Taking $a = 1$ and $t = b$
in~\eqref{eq:hforig}, and making the corresponding changes of variable and
parameters ($\nu = \alpha$, $\mu = \alpha+\beta+2n+1$, $\lambda = \beta$) we get
\begin{equation*}
  I_{-}(\alpha,\beta,n)(t)
  = \frac{\Gamma(\alpha+n+1)} {2^\beta \Gamma(\alpha+1) \Gamma(\beta+n+1)}
  {\,{}_2F_1}(\alpha+n+1,-n-\beta;\alpha+1;t^2),
\end{equation*}
which is valid for $\alpha>-1$ and $\beta>-1$ in the interval $0<t<1$.
Moreover, we have
\begin{equation*}
{\,{}_2F_1}(\alpha+n+1,-n-\beta;\alpha+1;t)
= (1-t)^\beta {\,{}_2F_1}(-n,\alpha+\beta+n+1;\alpha+1;t),
\end{equation*}
where $\alpha,\beta>-1$, $n = 0,1,2,\dots$, and
\begin{equation}
\label{eq:Pnab}
  P_n^{(\alpha,\beta)}(y) = \tfrac{\Gamma(n+\alpha+1)}{\Gamma(\alpha+1)
  \Gamma(n+1)} {\,{}_2F_1}(-n, \alpha+\beta+n+1; \alpha+1; \tfrac{1-y}2),
\end{equation}
whenever $\alpha, \beta> -1$ and $-1 < y < 1$. Therefore,
\begin{equation*}
  I_{-}(\alpha,\beta,n)(t)
  = 2^{-\beta} \, \tfrac{\Gamma(n+1)}{\Gamma(\beta+n+1)} \, (1-t^2)^\beta
  P^{(\alpha,\beta)}_n(1-2t^2), \quad t \in (0,1).
\end{equation*}
Now, we are going to evaluate
$I_{-}(\alpha,\beta,n)(t)$ for $t>1$. To do
that, let us take $t = a$, $b = 1$, $\nu = \alpha+\beta+2n+1$, $\mu = \alpha$, and
$\lambda = \beta$ in~\eqref{eq:hforig}. In this way,
$\frac12(\lambda+\mu-\nu+1) = 0,-1, -2, \dots$, so the coefficient
$1/\Gamma(\frac12(\lambda+\mu-\nu+1))$ vanishes and we get
$I_{-}(\alpha,\beta,n)(t) = 0$.

Finally, let us prove the second part of the lemma. To this end, we take, in~\eqref{eq:hforig},
$a = 1$ and $t = b$, with parameters $\lambda = -\beta$, $\mu
= \alpha+\beta+2n+1$ and $\nu = \alpha$. Then, for $\beta<1$, $\alpha
+\beta>-1$, and $0<t<1$ we get
\begin{equation*}
  I_{+}(\alpha,\beta,n)(t)
  = \frac{2^{\beta} \Gamma(\alpha+\beta+n+1)}{\Gamma(\alpha+1) \Gamma(n+1)} 
  {\,{}_{2}F_{1}}(\alpha+\beta+n+1,-n;\alpha+1;t^{2}).
\end{equation*}
Then, by using~\eqref{eq:Pnab}, \eqref{eq:HTbjnab}~follows.
\end{proof}

%------------------------------
\subsection{Proof of Lemma~\ref{lem:F-PQ}}
\label{sub:prooflemma}
%------------------------------

We start by evaluating $\calF_\alpha (\calJ_{\alpha+\beta,k})(t)$ for $\alpha>-1$ and
$\alpha+\beta>-1$.

By definition,
\begin{equation*}
  \calF_{\alpha}(\calJ_{\alpha+\beta,k})(t)
  = \frac12 \int_{\R} \frac{J_{\alpha+\beta+k+1}(x)}{x^{\alpha+\beta+1}}
  \left( \frac{J_{\alpha}(xt)}{(xt)^{\alpha}}
  - \frac{J_{\alpha+1}(xt)}{(xt)^{\alpha+1}}\,xti\right) |x|^{2\alpha+1}\,dx.
\end{equation*}
For the case $k = 2n$, by decomposing into even and odd functions, we can write
\begin{equation}
\label{eq:tpositive}
  \calF_{\alpha}(\calJ_{\alpha+\beta,2n})(t) =
  \int_{0}^{\infty} \frac{J_{\alpha+\beta+2n+1}(x)}{x^{\alpha+\beta+1}}
  \frac{J_{\alpha}(xt)}{(xt)^{\alpha}}\,x^{2\alpha+1}\,dx.
\end{equation}
Then, for $t>0$, by using~\eqref{eq:HTjnab} in Lemma~\ref{lem:HTjnab},
\eqref{eq:C2n} and~\eqref{eq:hC2n}, it
follows that
\begin{multline*}
  \calF_{\alpha}(\calJ_{\alpha+\beta,2n})(t)
  = t^{-\alpha}
  \int_{0}^{\infty} x^{-\beta} J_{\alpha+\beta+2n+1}(x)J_{\alpha}(xt)\,dx \\
  = \frac{\Gamma(n+1)}{2^{\beta} \Gamma(\beta+n+1)}
  \, (1-t^{2})^{\beta} P_{n}^{(\alpha,\beta)}(1-2t^2) \chi_{[0,1]}(t) \\
  = (-1)^n \, \frac{\Gamma(\alpha+\beta+1) \Gamma(n+1) \Gamma(\alpha+n+1)}{2^{\beta}
  \Gamma(\alpha+1)\Gamma(\beta+n+1)\Gamma(\alpha+\beta+n+1)}
  \, (1-t^{2})^{\beta} C_{2n}^{(\beta+1/2,\alpha+1/2)}(t) \chi_{[0,1]}(t) \\
  = \frac{i^k}{2^{\alpha+\beta+1}\Gamma(\alpha+\beta+1)(\alpha+\beta+k+1)}
  \, \calQ^{(\alpha,\beta)}_{k}(t)\chi_{[0,1]}(t).
\end{multline*}
For $t<0$, let us make, in~\eqref{eq:tpositive}, the
change $t_{1} = -t$, use the evenness of the function
$J_{\alpha}(z)/z^{\alpha}$, % and $J_{\alpha+1}(z)/z^{\alpha+1}$,
proceed as in the case $t>0$, and undo the change. Then, we get
\begin{equation*}
  \calF_{\alpha}(\calJ_{\alpha+\beta,2n})(t) 
  = \frac{i^k}{2^{\alpha+\beta+1}\Gamma(\alpha+\beta+1)(\alpha+\beta+k+1)}
  \, \calQ^{(\alpha,\beta)}_{k}(t)\chi_{[-1,0]}(t).
\end{equation*}
Thus, \eqref{eq:F-Q} for even $k$ is proved. The case $k=2n+1$ is completely similar.

Proceeding in the same way, the formula~\eqref{eq:F-P} follows from~\eqref{eq:HTbjnab}.

\subsection*{Acknowledgement}
The authors whish to thank the referee for careful reading of the manuscript, 
comments and the suggestion of Remark~\ref{rem:referee}.

%------------------------------

%------------------------------
\end{document}